\documentclass{article}
\usepackage{graphicx,amsmath,amsfonts,amssymb,amsthm,mathrsfs,bbm,stmaryrd,array, enumitem, array, caption, float, longtable}
\usepackage{hyperref}

\makeatletter
\def\th@plain{%
  \thm@notefont{}% same as heading font
  \itshape % body font
}
\def\th@definition{%
  \thm@notefont{}% same as heading font
  \normalfont % body font
}
\makeatother

\newtheorem{theorem}{Theorem}[section]
\newtheorem{proposition}[theorem]{Proposition}
\newtheorem{conjecture}[theorem]{Conjecture}
\newtheorem{lemma}[theorem]{Lemma}
\newtheorem{corollary}[theorem]{Corollary}

\begin{document}
\title{Cyclotomic Congruences and Lucas Sequences\footnotetext{\noindent Supported by National Natural Science Foundation of China, Project 12071421.}
\author{Tyler Ross\footnote{Corresponding author.}  \\ \small
\textit{School of Mathematical Sciences, Zhejiang University}\\
\small \textit{Hangzhou, Zhejiang, 310058, P.\ R.\ China} \\ \small \href{mailto:tylerxross@gmail.com}{\textit{tylerxross@gmail.com}}   
\and Zhongyan Shen \\ \small \textit{Department of Mathematics}, \\ \small \textit{Zhejiang International Studies University}\\
\small \textit{Hangzhou, Zhejiang, 310023, P.\ R.\ China} \\ \small \href{mailto:huanchenszyan@163.com}{\textit{huanchenszyan@163.com}}
\and Tianxin Cai \\ \small \textit{School of Mathematical Sciences, Zhejiang University}\\ \small \textit{Hangzhou, Zhejiang, 310058, P.\ R.\ China} \\ \small \href{mailto:txcai@zju.edu.cn}{\textit{txcai@zju.edu.cn}}}
\date{}
}

\maketitle
%------------------0.Abstract----------------------------------------

\begin{abstract}
    In this paper, we extend the $p$-adic valuations originally obtained by Carmichael for the sequences obtained by applying M\"obius inversion to Lucas sequences to $p$-adic congruences, from which we immediately derive corresponding congruences for Lucas sequences. As a corollary, we also establish some constraints on the entry point behavior of primes in Lucas sequences, on the basis of which we conjecture the presence of a strong Chebyshev-like bias in real regular Lucas sequences.
\end{abstract}

\noindent \textit{Keywords:} Lucas sequence, Sylvester sequence, M\"obius dual, cyclotomic polynomial, Fibonacci number \\

\noindent \textit{Mathematics Subject Classification 2020:} primary 11B50; secondary 11B39
%------------------1.Introduction----------------------------------------
\section{Introduction} \label{sec1}

Let $P, ~Q \in \mathbb{Z} \setminus(0)$ be any two nonzero integers such that also the discriminant $D=P^2 - 4Q$ of the polynomial $X^2-PX+Q \in \mathbb{Z}[X]$ is nonzero. Then the Lucas sequences \begin{eqnarray*} 
    U(P,Q) = (U_n(P,Q))_{n \geq 0}, \\
    V(P,Q) = (V_n(P,Q))_{n \geq 0},
\end{eqnarray*}
of the first and second kind respectively, in parameters $P, Q \in \mathbb{Z} \setminus(0)$ are the integer sequences given by the Binet forms
\begin{equation*}
    U_n = \frac{\alpha^n - \beta^n}{\alpha - \beta}, ~V_n = \alpha^n + \beta^n    
\end{equation*}
where
\begin{equation*}
    \alpha = \frac{P + \sqrt{D}}{2}, ~\beta = \frac{P-\sqrt{D}}{2}   
\end{equation*}
are the roots of the characteristic polynomial. It is well known that these sequences have the form of second-order linear recurrence sequences. In particular, when $P = 1$, $Q = -1$, we get the familiar Fibonacci numbers and Lucas numbers,
\begin{equation*}
    F = (F_n)_{n \geq 0} = (U_n(1,-1))_{n \geq 0} ~,~ L = (L_n)_{n \geq 0} = (V_n(1,-1))_{n \geq 0}.
\end{equation*}
In the following, we suppress all instances of the parameters $P$, $Q$ when these are taken to be arbitrary but fixed.

When $U, V$ are \textit{nondegenerate}, by which we mean that $U_n, V_n \neq 0$ for all $n \geq 1$, or equivalently that $\alpha/\beta$ is not a root of unity, we define the \textit{M\"obius dual sequences}
\begin{equation*}
    M^U_n = \prod_{d \mid n} U_d ^{\mu(n/d)}, ~M^V_n = \prod_{d \mid n} V_d ^{\mu(n/d)},
\end{equation*}
where $\mu: \mathbb{Z}_{>0} \longrightarrow \{-1,0,1\}$ is the M\"obius $\mu$-function. A straightforward calculation shows that for $n > 1$ the sequence $M^U$ can be written in terms of cyclotomic polynomials as
\begin{equation*}
    M^U_n = \beta^{\varphi(n)}\Phi_n(\alpha/\beta),
\end{equation*}
where $\varphi: \mathbb{N} \longrightarrow \mathbb{N}$ is Euler's totient function and $\Phi_n \in \mathbb{Z}[X]$ is the $n$-th cyclotomic polynomial. Such sequences are sometimes referred to as Sylvester sequences (see \cite{Syl}) and have been made use of extensively by Carmichael and others to study the divisibility properties of Lucas sequences (see \cite{Car}). It can be shown that $M^U_n \in \mathbb{Z}$ for all $n \geq 1$, while $M^V_n \in \mathbb{Z}$ for all odd $n \geq 1$ and at most finitely many even $n \geq 1$ (see \cite{Ross}).

We assume throughout the sequel that $U, V$ are nondegenerate. We do not, on the other hand, assume unless explicitly stated that they are \textit{regular}, that is, that $(P, Q)=1$.

In this paper, we extend the $p$-adic valuations obtained by Carmichael for the M\"obius duals of regular Lucas sequences to $p$-adic congruences, abandoning also the regularity requirement. As an immediate corollary, we derive $p$-adic congruences for $p$-adically adjacent terms in Lucas sequences, generalizing some previously known results. As a second corollary, we obtain constraints on the entry point behavior of primes in Lucas sequences, on the basis of which we conjecture that real regular Lucas sequences exhibit a strong Chebyshev-like bias.

\section{Main Results} \label{sec2}

We first fix some notation that we will make use of throughout the remainder of the paper. For $p$ prime, we write
\begin{equation*}
    z_U(p) = \text{min}(n \geq 1:p \mid U_n)
\end{equation*}
for the \textit{entry point}, or \textit{rank of apparition}, of $p$ in $U$. This number always exists, except in the case that $p \nmid P$ and $p \mid Q$. When $n=z_U(p)$, we say that $p$ is a \textit{characteristic factor} of $U_n$. We write
\begin{equation*}
    \partial_p(m)=m/p^{v_p(m)}
\end{equation*}
for the $p$-free part of $m \in \mathbb{Z} \setminus (0)$.

\begin{theorem} \label{modpk-thm}
Fix any positive integers $p, n \geq 1$, with $p$ prime and $(p,n)=1$. If $p \nmid (P,Q)$, we have the following congruences.

\begin{enumerate}[label=(\alph*)] 
\item \label{mpk1} If $p \mid D$, then 
\begin{equation*}
M^U_{p^k n} \equiv
\begin{cases}
p \pmod{p^k}, &\text{\emph{ if }} n=1,\\
1 \pmod{p^k}, &\text{\emph{ if }} n > 1,
\end{cases}
\end{equation*}
for all $k \geq 1$. If $p = 2$, then we have stronger congruences
\begin{equation*}
M^U_{2^kn} \equiv
\begin{cases}
2 \pmod{2^{k+1}}, &\text{\emph{if }} n = 1, k \geq 3 \\
1 \pmod{2^{k+1}}, &\text{\emph{if }} n > 1, k \geq 2;
\end{cases}
\end{equation*}
if $p > 2$, $n = 1$, then
\begin{equation*}
M^U_{p^k} \equiv p \pmod{p^{k+1}}
\end{equation*}
for all $k \geq 2$, if $p = 3$, and all $k \geq 1$, if $p > 3$.

\item  \label{mpk2} If $p \nmid D$, then
\begin{equation*}
        M^U_{p^k n} \equiv
        \begin{cases}
            \left( \frac{D}{p} \right) &\pmod{p^k}, \text{\emph{ if } } n =1, \\
            1 &\pmod{p^k}, \text{\emph{ if }} n > 1,        \end{cases}
    \end{equation*}
    for all $k \geq 1$, where $\left(\frac{~}{~} \right)$ is the Kronecker symbol, unless $n = z_U(p)$, in which case
    \begin{equation*}
        M^U_{p^k z_U(p)} \equiv \left(\frac{D}{p} \right)p \pmod{p^k}
    \end{equation*}
    for all $k \geq 1,$
    with stronger congruences
    \begin{equation*}
        M^U_{p^kz_U(p)} \equiv \left(\frac{D}{p} \right)p \pmod{p^{k+1}} 
    \end{equation*}
    for all $k \geq 2$, if $p = 2$, and all $k \geq 1$, if $p > 2$.
\end{enumerate}
If $p \mid (P, Q)$, then instead we have the following congruences.
\begin{enumerate}[label=(\alph*)] \setcounter{enumi}{2}
    \item \label{mpk4} For all $k \geq 1$,
    \begin{equation*}
        M^U_{p^kn} \equiv 0 \pmod{p^{p^{k-1}}},
    \end{equation*}
unless $2v_p(P) > v_p(Q)$, $p = 2$, $n = 1$, in which case we have only
\begin{equation*}
    M^U_{2^k} \equiv 0 \pmod{2^{2^{k-2}+1}}
\end{equation*}
for $k \geq 2$.
\end{enumerate}

\end{theorem}

\begin{theorem} \label{vmodpk-thm}
    Fix any positive integers $p, n \geq 1$ with $p$ prime and $(p,n)=1$; we also assume that $p \nmid (P,Q)$. Then
    \begin{equation*}
        M^V_{p^k n} \in \mathbb{Z}_{(p)} = \left\{r \in \mathbb{Q}: v_p(r) \geq 0 \right\},
    \end{equation*}
    for all $k \geq 1$ unless $p = 2$ divides $D$, $n = 1$, $k = 1$, or $p=2$ does not divide $D$, $n=3$, $k = 1$, or $p > 2$ does not divide $D$, and $n = z_U(p)$ is even; we may or may not have $M^V_{2^kn} \in \mathbb{Z}_{(2)}$ in the former two cases, and never have $M^V_{p^k n} \in \mathbb{Z}_{(p)}$ in the latter case. Excluding these exceptions,
    \begin{equation*}
        M^V_{p^k n} \equiv 1 \pmod{p^k}
    \end{equation*}
for all $k \geq 1$, unless $p = 2$ divides $D$, $n=1$, $k = 2$, in which case
\begin{equation*}
M^V_{4} \equiv \pm1 \pmod{4},    
\end{equation*}
or unless $p > 2$, $n = z_U(p)/2$, in which case
\begin{equation*}
    M^V_{p^k z_U(p)/2} \equiv \left(\frac{D}{p} \right)p \pmod{p^{k}}
\end{equation*}
for all $k \geq 1$, with stronger congruences
\begin{equation*}
    M^V_{p^k z_U(p)/2} \equiv \left(\frac{D}{p} \right)p \pmod{p^{k+1}}
\end{equation*}
for all $k \geq 1$, if moreover $n=z_U(p)/2$ is odd.
\end{theorem}

 Applying the $p$-adic congruences in Theorems \ref{modpk-thm} and \ref{vmodpk-thm} at each prime dividing $n > 1$, we readily obtain congruences for the sequences $M^U_n$, $M^V_n$ modulo $n$ for all $n \geq 1$ (provided in the latter case that $M^V_n \in \mathbb{Z}_{(n)}$, for which condition we give a complete characterization); we can assume that $n \neq p^k$ has at least two distinct prime factors, the situation at prime powers having already been completely determined. We also consider here only the case that $(n,(P,Q))=1$.

\begin{corollary} \label{modn-thm}
Suppose $n > 1$ has at least two distinct prime factors, with $(n,(P,Q))=1$, and let $p$ be the largest prime dividing $n$.
\begin{enumerate}[label=(\alph*)]
\item \label{modn1} Then
\begin{equation*}
    M^U_n \equiv
    \begin{cases}
        \left(\frac{D}{p} \right)p &\pmod{n},  \text{\emph{ if }} \partial_{p}(n) = z_U(p), \\
        1 &\pmod{n}, \text{\emph{ otherwise}},
    \end{cases}
\end{equation*}
unless $2 \nmid PQ$ and $n = 2^k \cdot 3$ $(k \geq 1)$, in which case
\begin{align*}
    M^U_6 &\equiv
    \begin{cases}
        0 &\pmod{6}, \text{\emph{ if }} z_U(3)=2, \\ 
        4 &\pmod{6}, \text{\emph{ if }} 3 \mid Q, \text{\emph{ or }} z_U(3)= 3 \text{\emph{ or }} 4,
    \end{cases}, \\
    M^U_{12} &\equiv
    \begin{cases}
        10 &\pmod{12}, \text{\emph{ if }} 3 \mid Q, \text{\emph{ or }} z_U(3)=2 \text{\emph{ or }} 3, \\ 
        6 &\pmod{12}, \text{\emph{ if }} z_U(3)=4,
    \end{cases} \\
    M^U_{2^k \cdot 3} &\equiv 2 \left(\frac{D}{2} \right) \pmod{2^k \cdot 3}, \text{\emph{ if }} k \geq 3.
\end{align*}
\item \label{modn2} Similarly,
\begin{equation*}
    M^V_n \equiv 
    \begin{cases}
        \left(\frac{D}{p} \right)p &\pmod{n}, \text{\emph{ if }} \partial_p(n) = z_U(p)/2, \\
        1 &\pmod{n}, \text{\emph{ otherwise}},
    \end{cases}
\end{equation*}
unless $n = 6$, or $\partial_p(n) = z_U(p)$ is even, in which case it does not hold in general that $M^V_n \in \mathbb{Z}_{(n)}$.
\end{enumerate}
\end{corollary}

The following $p$-adic congruences for $p$-adically adjacent Lucas sequence terms follow immediately from the corresponding congruences for $M^U$, $M^V$. Congruences of the forms \ref{lmp1} and \ref{lmp1.5} were first observed by Young (\cite{Young}), in some cases with a weaker modulus.

\begin{corollary} \label{lift-thm}
Fix any positive integers $p, n \geq 1$ with $p$ prime and $(p, n)= 1$. If $p \nmid (P,Q)$, we have the following congruences.
    \begin{enumerate} [label = (\alph*)]
        \item \label{lmp1} If $p \mid D$, then 
        \begin{equation*} 
            \frac{U_{p^k n}}{U_{p^{k-1}n}} \equiv p \pmod{p^k}
        \end{equation*}
    for all $k \geq 1$, with stronger congruences
        \begin{equation*} 
            \frac{U_{p^k n}}{U_{p^{k-1}n}} \equiv p \pmod{p^{k+1}}
        \end{equation*}
    for all $n \geq 1$, if $p = 2$ and $k \geq 3$, or $p = 3$ and $k \geq 2$, or $p \geq 5$ and $k \geq 1$. 
    \item \label{lmp1.5} If $p \nmid D$, then 
        \begin{equation*} 
            \frac{U_{p^k n}}{U_{p^{k-1}n}} \equiv \left(\frac{D}{p} \right) \pmod{p^{k}}
        \end{equation*}
        for all $k \geq 1$, unless $z_U(p) \mid n$, in which case
        \begin{equation*} 
            \frac{U_{p^k n}}{U_{p^{k-1}n}} \equiv p \pmod{p^k}
        \end{equation*}
        for all $k \geq 1$, with stronger congruences
        
        \begin{equation*} 
            \frac{U_{p^k n}}{U_{p^{k-1}n}} \equiv p \pmod{p^{k+1}}
        \end{equation*}
    if $p = 2$ and $k \geq 2$, or $p > 2$ and $k \geq 1$.
        \item \label{lmp2} Except for the cases $p = 2$ divides $D$, or $p = 2$ does not divide $D$, and $3 \mid n$, $k = 1$, or $p > 2$ does not divide $D$ and $z_U(p) \mid n$, we have
        \begin{equation*}
            \frac{V_{p^k n}}{V_{p^{k-1}n}} \in \mathbb{Z}_{(p)}, ~\frac{V_{p^k n}}{V_{p^{k-1}n}} \equiv 1 \pmod{p^k}
        \end{equation*}
        for all $k, n \geq 1$ and any prime $p$; in the exceptional cases, we do not \emph{a priori} have $V_{p^k n}/V_{p^{k-1}n} \in \mathbb{Z}_{(p)}$, but it is always the case that
        \begin{equation*}
            V_{p^k n} \equiv V_{p^{k-1}n} \pmod{p^k}
        \end{equation*}
        for all $k, n \geq 1$ and any prime $p$.
    \end{enumerate}
    If $p \mid (P, Q)$, we have the following congruences.
    \begin{enumerate}[label=(\alph*)] \setcounter{enumi}{3}
    \item \label{lmp3} For all $k \geq 1$,
    \begin{equation*}
         \frac{U_{p^k n}}{U_{p^{k-1}n}} \equiv 0 \pmod{p^{p^{k-1}}},
    \end{equation*}
unless $2v_p(P) > v_p(Q)$, $p = 2$, $n = 1$, in which case we have only the weaker congruence
\begin{equation*}
    \frac{U_{2^k}}{U_{2^{k-1}}} \equiv 0 \pmod{2^{2^{k-2}+1}}
\end{equation*}
for $k \geq 2$.

\end{enumerate}

\end{corollary}

As a second application, we obtain from the congruences in Theorem \ref{modpk-thm} and Corollary \ref{modn-thm} constraints on the entry point behavior of certain primes in regular Lucas sequences.

\begin{corollary} \label{mult-cor}
Let $U$ be a regular Lucas sequence, fix $n > 1$ and let $p$ be the largest prime dividing $n$. 
\begin{enumerate}[label = (\alph*)]
    \item If $p \nmid D$ and $\partial_p(n) = 1 \text{\emph{ or }} \partial_p(n) =z_U(p)$, then
    \begin{equation*}
        \prod_{z_U(q)=n}\left(\frac{D}{q} \right)^{v_q(U_n)}=\left(\frac{D}{p}\right) \times \text{\emph{sign}}(M^U_n),
    \end{equation*}
    where $\text{\emph{sign}}(m)=m/\left|m\right|$ for $m \in \mathbb{Z} \setminus (0)$, the product running through all characteristic factors $q$ of $U_n$, except possibly if $n=p=2$, or if $p \mid P$ and $n=2p$. \item Otherwise,
    \begin{equation*}
        \prod_{z_U(q)=n}\left(\frac{D}{q}\right)^{v_q(U_n)}=\text{\emph{sign}}(M^U_n),
    \end{equation*}
    except possibly if $p=2$ divides $D$, and $n = 2 \text{ \emph{or} } 4$, or $p > 2$ divides $D$, and $n = p$, or if $2 \nmid PQ$ and $n = 2^k \cdot 3$ $(k \geq 1)$; in the latter case,
    \begin{equation*}
        \prod_{z_U(q)=2^k \cdot 3}\left(\frac{D}{q}\right)^{v_q(U_{2^k \cdot 3})}=\left(\frac{D}{2} \right) \times \text{\emph{sign}}(M^U_{2^k \cdot 3})
    \end{equation*}
    for all $k \geq 3$.
\end{enumerate}
\end{corollary}

Here it is worthwhile to illustrate by way of concrete examples the constraints imposed by this result. We consider the Fibonacci numbers $F = U(1,-1)$. It is clear that all sign considerations in the preceding corollary vanish in this case (and more generally if $D > 0$).

Note that a prime number satisfying $v_p(F_{z_F(p)}) > 1$ is called a Wall-Sun-Sun prime (see \cite{Sun1}). It is not known whether or not any such primes exist, although it has been established that there are no Wall-Sun-Sun primes smaller than $9.7 \times10^{14}$ (see \cite{Dor}). If every characteristic factor of $F_n$ is smaller than this bound, or under the hypothesis that no Wall-Sun-Sun primes exist, then all of the statements about multiplicities in the following corollary can be deleted.

\begin{corollary} \label{fib-cor}
Let $p ,n \geq 1$, with $p$ prime.
\begin{enumerate}[label = (\alph*)]
    \item \label{fib1} If $p = 5$ or $\left(\frac{p}{5} \right)=1$, then $F_{p^k}$ has an even number, possibly zero, of characteristic factors satisfying $\left(\frac{q}{5} \right) = -1$, including multiplicities, for all $k \geq 1$. The smallest prime power $n=p^k$ such that $F_n$ has a nonzero even number of characteristic factors with negative Legendre symbol is $361=19^2$, with characteristic factors
    \begin{align*}
        &6567762529, 1196762644057, 3150927827816930878141597, \\ & \text{and } 12020126510714734783009241.
    \end{align*}
    \item \label{fib2} If $\left(\frac{p}{5} \right)=-1$, then $F_{p^k}$ has an odd number of characteristic factors satisfying $\left(\frac{q}{5} \right) = -1$, including multiplicities, for all $k \geq 1$, with the exception of $F_2 = 1$; in particular, there is always at least one.
    \item \label{fib4} If $n = 2^i 3^j$ with both $i, j \geq 1$, then $F_{n}$ has an odd number of characteristic factors $\left(\frac{q}{5} \right) = -1$ if and only if $i = 2$, $j > 1$ or $i > 2, ~j=1$. The first occurrence among such indices of a nonzero even number of characteristic factors with negative Legendre symbols is at $F_{216}$, which has characteristic factors $6263$ and $177962167367$.
    \item \label{fib5} If $n = 2^i 3^j 5^k$ with $k \geq 1$, $i, j \geq 0$, then $F_n$ has an even number of characteristic factors satisfying $\left( \frac{q}{5} \right) = -1$. In fact, there are no indices $n \leq 1000$ divisible by $5$ such that $F_n$ has a characteristic factor with negative Legendre symbol.
    \item  \label{fib6} If $p > 5$ is the largest prime dividing $n$, then $F_n$ has an odd number of characteristic factors $\left(\frac{q}{5} \right)=-1$, including multiplicities, if and only if $\left(\frac{p}{5} \right) = -1$ and $\partial_p(n)=1$ or $\partial_p(n) = z_F(p)$.
\end{enumerate}
\end{corollary}

\section{A Conjecture} \label{sec2.5}

Corollary \ref{mult-cor} suggests the presence of a Chebyshev-like bias for the entry points of primes in Lucas sequences, in particular real regular Lucas sequences, which we formalize as a conjecture, after a brief discussion of what is known about the distribution behavior of these entry points. It is convenient to set $z_U(p) = \infty$ for primes dividing $Q$ but not $P$, so as to obtain a map
\begin{equation*}
z_U: \mathbb{P} \longrightarrow \mathbb{N} ~\cup ~ \left\{ \infty \right\}
\end{equation*}
defined on the set $\mathbb{P}$ of all primes. A few basic observations: the fact that every prime except those dividing $Q$ but not $P$ eventually appears in $U$ means that the set
\begin{equation*}
\left\{p \in \mathbb{P}: z_U(p) < \infty \right\}
\end{equation*}
has finite complement in $\mathbb{P}$; similarly, Carmichael's theorem and its extension to negative discriminants by Bilu, Hanrot, and Voutier (see \cite{Car} and \cite{Bilu}) show that the image $z_U(\mathbb{P})$ is cofinite in $\mathbb{N}$.

Lagarias showed (\cite{Lag}) that the set
\begin{equation*}
\left\{p \in \mathbb{P}: z_F(p) \text{ is even} \right\}
\end{equation*}
has asymptotic density $\frac{2}{3}$ in $\mathbb{P}$; another way to say the same thing is that $\frac{2}{3}$ of all primes divide some term of $L = V(1,-1)$. Subsequently, Cubre and Rouse proved (\cite{Cub}) a conjecture due to Bruckman and Anderson (see \cite{Bruck}), extending this result to establish asymptotic densities for sets of the form
\begin{equation*}
\left\{p \in \mathbb{P}: z_F(p) \equiv 0 \pmod{m} \right\}
\end{equation*}
for arbitrary integers $m \geq 1$.

In order to state our conjecture, we introduce some notation. Fix any Lucas sequence $U$, and let
\begin{align*}
R_U &= \{p \in \mathbb{P}: \left(\frac{D}{p} \right) = 1 \}, \\
N_U &= \{p \in \mathbb{P}: \left(\frac{D}{p} \right) = -1 \}.
\end{align*}
For $x > 0$, define
\begin{align*}
Z_U(x) = \left\{p \in \mathbb{P}: z_U(p) \leq x \right\} = \{p \in \mathbb{P}: \prod_{1 \leq n \leq  x }U_n \equiv 0 \pmod{p} \},
\end{align*}
and set
\begin{align*}
Z^R_U(x) &= Z_U(x) ~ \cap ~ R_U, \\
Z^N_U(x) &= Z_U(x) ~ \cap ~ N_U.
\end{align*}
Finally, we define the bias term
\begin{equation*}
B_U(x) = \#\{1 \leq n \leq x: \#Z^N_U(n) < \#Z^R_U(n) \}.
\end{equation*}

\begin{conjecture}
\begin{enumerate}[label = (\alph*)]
\item[]
\item (Weak) If $U$ is a regular Lucas sequence with $D > 0$, then
\begin{equation*}
    B_U(x)/x \underset{x \rightarrow \infty}\longrightarrow 1.
\end{equation*}
\item (Strong) If $U$ is a regular Lucas sequence with $D > 0$, then
\begin{align*}
    \#Z^N_U(n) < \#Z^R_U(n)    
\end{align*}
for all but finitely many $n \geq 1$; for example, we conjecture that 
\begin{align*}
    \#Z^N_F(n) < \#Z^R_F(n)    
\end{align*}
for all $n > 36$.
\end{enumerate}
\end{conjecture}

It would be valuable to have both data and theory related to the sizes of the sets $Z^R_U(x)$, $Z^N_U(x)$, and their ratios and differences, across various Lucas sequences, as well as various related quantities, such as the total number of primes, with positive or negative Kronecker symbol respectively, appearing in a Lucas sequence at indices $n \leq x$, the number of indices $n \leq x$ admitting a characteristic factor with positive or negative Kronecker symbol respectively, and so on. As a small first step in this direction, we summarize in Table \ref{tab1} and Figure \ref{fig1} some relevant data for the first thousand Fibonacci numbers. The factorizations of these numbers are available at Blair Kelly's page of Fibonacci and Lucas number factorizations (see \cite{Kelly}), which also includes partial factorizations (and some complete factorizations) for all indices up to $n < 10000$ (see also \cite{Alf1} and \cite{Alf2} for tables of Fibonacci entry points for all primes $p < 100000$).

\noindent
\begin{longtable}{|c|c|c|}
\caption{Bias in Fibonacci entry points.} \label{tab1} \\ \hline
\textbf{Indices} & \multicolumn{2}{c|}{\textbf{Entry Point Behavior}} \\ \cline{2-3}
         & \(\# Z_F^R(n)\) & \(\# Z_F^N(n)\) \\ \hline
\endfirsthead

\multicolumn{3}{c}%
{{\bfseries Table \thetable\ (Continued)}} \\ \hline
\textbf{Indices} & \multicolumn{2}{c|}{\textbf{Entry Point Behavior}} \\ \cline{2-3}
         & \(\#Z_F^R(n)\) & \(\#Z_F^N(n)\) \\ \hline
\endhead

\hline \multicolumn{3}{r}{{Continued on the next page.}} \\ \hline
\endfoot

\hline
\endlastfoot

\(n = 1\)        & 0       &  0 \\ \hline
\(n = 2\)        & 0       &  0 \\ \hline
\(n = 3\)        & 0       &  1 \\ \hline
\multicolumn{1}{|c|}{\(4 \le n \le 28\)}  & \multicolumn{2}{c|}{\(\#Z_F^R(n) < \#Z_F^N(n)\)} \\ \hline
\(n = 29\)       & 13      & 13 \\ \hline
\(n = 30\)       & 14      & 13 \\ \hline
\(n = 31\)       & 14      & 15 \\ \hline
\(n = 32\)       & 14      & 16 \\ \hline
\(n = 33\)       & 15      & 16 \\ \hline
\(n = 34\)       & 16      & 16 \\ \hline
\(n = 35\)       & 17      & 16 \\ \hline
\(n = 36\)       & 17      & 17 \\ \hline
\multicolumn{1}{|c|}{\(37 \le n \le 999\)} & \multicolumn{2}{c|}{\(\# Z^R_F(n) > \#Z^N_F(n)\)} \\ \hline
\(n = 1000\)     & 1970    & 959 \\ \hline

\end{longtable}

\begin{figure}[ht]
\centering
\includegraphics[width=\linewidth]{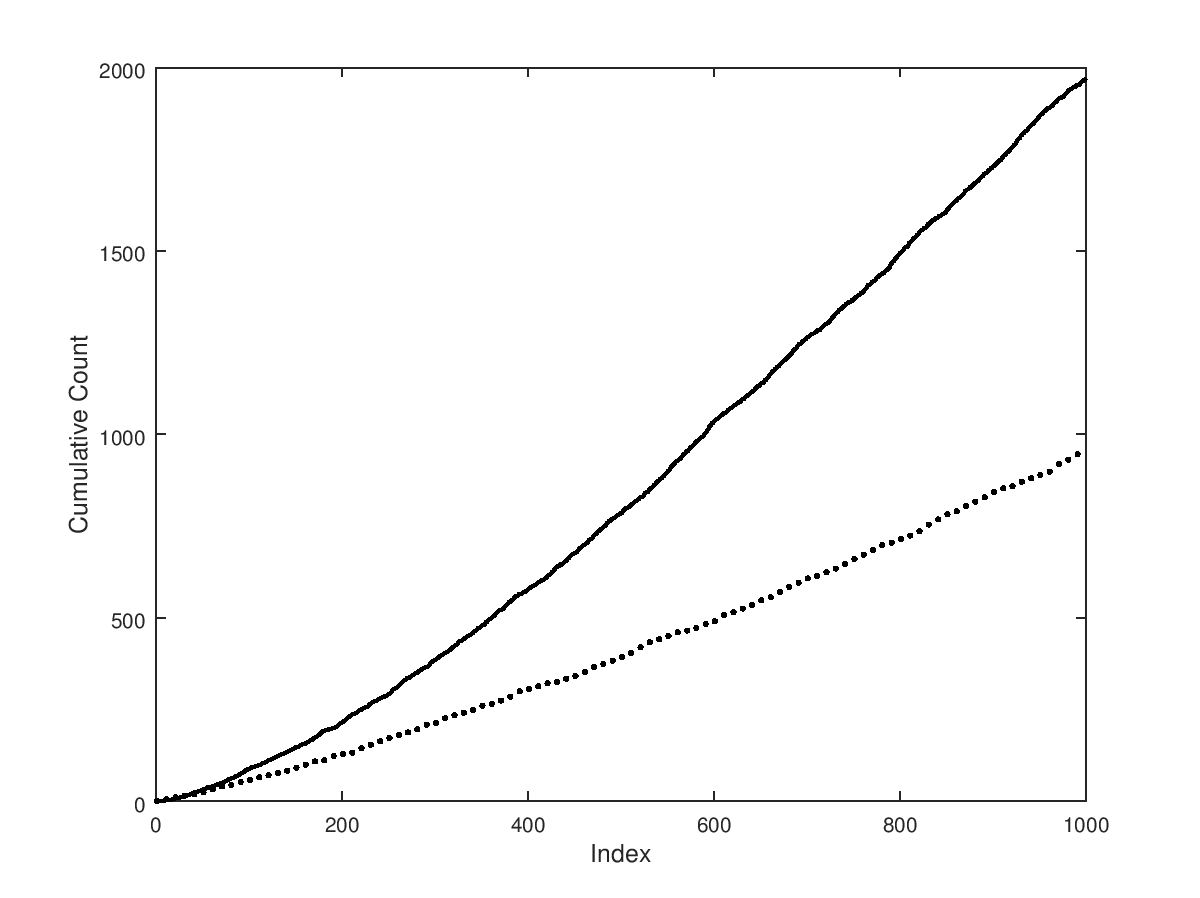} 
\caption{Bias in Fibonacci entry points. Here, $\# Z^R_F(n)$ is the solid line, $\# Z^N_F(n)$ the dotted line.}  
\label{fig1} 
\end{figure}

\section{Auxiliary Results} \label{sec3}
In this section, we gather some auxiliary results that we will need for the proofs of the theorems, starting with a few facts aboutthe sequences $M^U, M^V$.

\begin{proposition}[Doubling formula, \cite{Ross}] \label{dbl-thm}
For $n \geq 1$,
\begin{equation*} \label{dbleq0}
    M^{U}_{2n} = 
    \begin{cases}
        
        M^{V}_{n}, &\text{ \emph{if} } n \text{ \emph{is odd}.} \\
        M^{V}_{n}M^{U}_{n}, &\text{ \emph{if} } n \text{ \emph{is even},}
    \end{cases}
\end{equation*}
\end{proposition}
\begin{proposition}[\cite{Ross}] \label{vpM-thm}
    If $p$ is prime and $p \nmid (P,Q)$, we have the following valuations.
    \begin{enumerate}[label = (\alph*)]
    \item \label{vpM0} If $p \mid Q$, then $v_p(M^U_n)=0$ for all $n \geq 1$.
    \item \label{vpM1} If $p \mid D$, then
    \begin{equation*} \label{vpMeq1}
        v_p(M^U_n) =
        \begin{cases}
         v_p(U_p), &\text{\emph{if} } n = p, \\
         1, &\text{\emph{if} } n = p^k,~ k >1, \\
         0, &\text{\emph{otherwise}}.
        \end{cases}
    \end{equation*}
    \item \label{vpM2} If $p \nmid QD$, then
    \begin{equation*}
        v_p(M^U_n) =
        \begin{cases}
            v_p(U_{z_U(p)}), &\text{\emph{if }} n=z_U(p), \\
            v_p(U_{pz_U(p)})-v_p(U_{z_U(p)}), &\text{\emph{if }} n = pz_U(p), \\
            1, &\text{\emph{if }} n =p^kz_U(p),~k >1, \\
            0, &\text{\emph{otherwise.}}
        \end{cases}
    \end{equation*}
    \end{enumerate}
    If $p \mid (P, Q)$, then we have the following cases.
    \begin{enumerate} [label = (\alph*)] \setcounter{enumi}{3}
        \item \label{vpm3} If $v_p(Q) \geq 2v_p(P)$, then
        \begin{equation*}
            v_p(M^U_n)=
            \begin{cases}
                0, &\text{\emph{if} } n = 1, \\
                \varphi(n)v_p(P)+v_p(M^{U^{(p)}}_n), &\text{\emph{if }} n > 1,
            \end{cases}
        \end{equation*}
        where $U^{(p)}=U(\partial_p(P),\partial_p(Q))$.
        \item \label{vpm4} If $2v_p(P) > v_p(Q)$, then 
        \begin{equation*}
            v_p(M^U_n) =
            \begin{cases}
                v_p(P), &\text{\emph{if n = 2}},\\
                \left(\varphi(n)/2\right)v_p(Q)+1, &\text{\emph{if }} n = 2p^k, ~ p \text{\emph{ prime }},~k\geq 1, \\
                \lfloor \varphi(n)/2 \rfloor v_p(Q), &\text{\emph{otherwise,}}
            \end{cases}
        \end{equation*}
        unless $2v_p(P)=v_p(Q)+1$, $p = 2$ or $3$, $n = 2p$ in which case
        \begin{align*}
            v_p(M^U_{2p})=v_p(Q)+1+v_p(\partial_p(P)^2-\partial_p(Q)). 
        \end{align*}
    \end{enumerate}
\end{proposition}

We recall also a few basic facts concerning cyclotomic polynomials.

\begin{lemma}[\cite{Leb}, \cite{Lehm}, \cite{Wash}] \label{cycl-lem}
Let $\Phi_n \in \mathbb{Z}[X]$ be the $n$-the cyclotomic polynomial $(n \geq 1)$.
\begin{enumerate} [label = (\alph*)]
    \item \label{cycl1} Fix any positive integers $p, k, n \geq 1$ with $p$ prime and $(p,n)=1$. Then
    \begin{align*}
        \Phi_{p^kn}(X)=\Phi_n(X^{p^k})/\Phi_n(X^{p^{k-1}});
    \end{align*}
    in particular, if $n = 1$, then
    \begin{align*}
        \Phi_{p}(X) &=\sum_{j=0}^{p-1}X^j, \\
        \Phi_{p^k}(X) &= \Phi_p(X^{p^{k-1}}).
    \end{align*}
    \item \label{cycl1.5} For $n > 1$, $\Phi_n(X^{-1})=X^{-\varphi(n)}\Phi_n(X)$.
    \item If $\mathcal{K}=\mathbb{F}_q$ is a finite field, $\zeta \in \mathcal{K}^{\times}$, and $p = \text{\emph{char} }\mathcal{K}$ does not divide $n \geq 1$, then $\Phi_n(\zeta)=0$ if and only if $n = \text{\emph{ord}}_{\mathcal{K}^{\times}}(\zeta)$ is the order of $\zeta$ in $\mathcal{K}^{\times}$; in other words, if and only if $\zeta$ is a primitive $n$-th root of unity over $\mathbb{F}_{p}$ (here, and throughout, we interpret $\Phi_n$ as its canonical image in $\mathcal{K}[X]$).
\end{enumerate}
\end{lemma}

The proof of Theorem \ref{modpk-thm} essentially amounts to evaluating cyclotomic polynomials modulo prime ideal powers of a quadratic number field. We record here the rational integer version of these calculations, which applies directly to the split prime case, and serves as the template for the arguments in the case of ramified or inert primes. This result is quite straightforward, but we did not find a suitable reference, and present a proof below. For $p, \zeta \in \mathbb{Z}$ with $p$ prime and $(p,\zeta)=1$, we write
\begin{equation*}
    \text{ord}_p(\zeta) = \text{min}(n \geq 1: \zeta^n \equiv 1 \pmod{p})
\end{equation*}
for the multiplicative order of $\zeta$ modulo $p$.

\begin{lemma} \label{ratcon-lem}
    Fix any $p, n, \zeta \in \mathbb{Z}$ with $p$ prime, $n \geq 1$, and $(p,\zeta)=(p,n)=1$.
    \begin{enumerate}[label = (\alph*)]
        \item \label{rc1} If $n \neq \text{\emph{ord}}_p(\zeta)$, then
        \begin{equation*}
            \Phi_{p^kn}(\zeta) \equiv
            \begin{cases}
                1 \pmod{p^{k}}, &\text{\emph{if} } p >2, \text{\emph{} or } p = 2, k = 1, \\
                1 \pmod{p^{k+1}}, &\text{\emph{if} } p =2, ~k >1.
            \end{cases}
        \end{equation*}
        \item \label{rc2} If $n=\text{\emph{ord}}_p(\zeta)$, then
        \begin{equation*}
            \Phi_{p^kn}(\zeta) \equiv
            \begin{cases}
                p \pmod{p^k}, &\text{\emph{if} } p = 2, k = 1, \\
                p \pmod{p^{k+1}}, &\text{\emph{if} } p > 2 \text{\emph{ or }} p=2, ~ k > 1.
            \end{cases}
        \end{equation*}
    \end{enumerate}
\end{lemma}

\begin{proof}
\ref{rc1} If $n \neq \text{ord}_p(\zeta)$, then $\Phi_{n}(\zeta) \not\equiv 0 \pmod {p}$, and therefore
\begin{equation*}
    \Phi_{n}(\zeta^{p^{k-1}}) \not\equiv 0 \pmod {p^k}    
\end{equation*}
for all $k \geq 1$; we get
\begin{equation*}
    \Phi_{p^kn}(\zeta)= \Phi_n(\zeta^{p^k})\Phi_n(\zeta^{p^{k-1}})^{-1} \equiv \Phi_n(\zeta^{p^k})\Phi_n(\zeta^{p^k})^{-1}=1 \pmod{p^k}.
\end{equation*}
When $p = 2$, we have $\zeta \equiv 1 \pmod{2}$, from which the stronger congruences
\begin{equation*}
    \Phi_n(\zeta^{2^k})\Phi_n(\zeta^{2^{k-1}})^{-1} \equiv 1 \pmod{2^{k+1}}
\end{equation*}
for $k > 1$ follow immediately.

\ref{rc2} If $p=2$, so that $\zeta \equiv 1 \pmod{2}$, $n=1$, we have $\Phi_{2^k}(X)=X^{2^{k-1}}+1$, and we are done. Suppose $p > 2$; it follows from $n = \text{ord}_p(\zeta)$ that $n = \text{ord}_{p^k}(\zeta^{p^{k-1}})$ for each $k \geq 1$, so that we have a factorization (not, in general, unique)
    \begin{equation*}
        \Phi_n(X) \equiv \prod_{(m,n)=1}(X-\zeta^{p^k m}) \pmod{p^{k+1}\mathbb{Z}[X]},
    \end{equation*}
    the product running over any reduced residue system modulo $n$. Therefore
    \begin{align*}
        \Phi_{p^k n}(X) &= \frac{\Phi_n(X^{p^k})}{\Phi_n(X^{p^{k-1}})} \\ &\equiv \prod_{(m,n)=1}\frac{X^{p^k}-\zeta^{p^km}}{X^{p^{k-1}}-\zeta^{p^km}} \\ &\equiv \prod_{(m,n)=1}\frac{X^{p^k}-\zeta^{p^{k+1}m}}{X^{p^{k-1}}-\zeta^{p^km}} \\ &= \prod_{(m,n)=1}\sum_{j=0}^{p-1}\zeta^{p^kmj}X^{p^{k-1}(p-1-j)} \pmod{p^{k+1}\mathbb{Z}[X]}, 
    \end{align*}
    from which it follows that
    \begin{equation*}
        \Phi_{p^kn}(\zeta) \equiv \zeta^{\varphi(p^kn)}\prod_{(m,n)=1}\Phi_{p^{k}}(\zeta_m) \pmod{p^{k+1}},
    \end{equation*}
    where $\zeta_m=\zeta^{pm-1} \equiv \zeta^{m-1} \pmod{p}$. It is easy to see that
    \begin{equation*}
        \Phi_{p^k}(\zeta_m) \equiv p \pmod{p^{k+1}}    
    \end{equation*}
    for $m \equiv 1 \pmod{n}$; if $m \not\equiv 1 \pmod{n}$, then $\text{ord}_{p}(\zeta_m) > 1$, and therefore $\Phi_{p^k}(\zeta_m) \equiv 1 \pmod{p^k}$ by \ref{rc1}. Since obviously $\zeta^{\varphi(p^{k}n)} \equiv 1 \pmod {p^k}$, we are done.
\end{proof}

\section{Proofs of Main Results}

\begin{proof}[Proof of Theorem \ref{modpk-thm}]

We work with
\begin{equation*} \label{mcycl-eq}
M_{p^k n} = \beta^{\varphi(p^k n)}\Phi_{p^kn}(\zeta),
\end{equation*}
for $k, n \geq 1$, where $\zeta = \alpha\beta^{-1}$. It will be clear in the course of the proof that $\zeta$ is $p$-integral in $\mathcal{K}=\mathbb{Q}[\sqrt{D}]$ whenever this is necessary.

\ref{mpk1} For $p \mid D$, let $\mathfrak{p}=\left(p, \sqrt{D} \right)$ be the unique prime of $\mathcal{K}$ lying over $p$. Since $\alpha - \beta = \sqrt{D}$, we see that $\alpha \equiv \beta \pmod{\mathfrak{p}}$, and $\alpha\beta=Q$ implies that both $\alpha$, $\beta$ are units modulo $\mathfrak{p}$; so $\zeta \equiv 1 \pmod{\mathfrak{p}}$. If $p = 2$, note moreover that $2 \mid D$ implies $2 \mid P$, in which case we have $4 \mid D$. It follows that, in this case, actually $\alpha \equiv \beta \pmod{2\mathcal{O}_{\mathcal{K}}}$, $\zeta \equiv 1 \pmod{2\mathcal{O_K}}$. Applying the identity $v_{\mathfrak{p}}(p)=2$, it is straightforward to obtain the following refinement of Lemma \ref{ratcon-lem}: if $n = 1$, then
\begin{equation*}
\Phi_{p^k}(\zeta) \equiv
\begin{cases}
0 \pmod{\mathfrak{p}^2}, &\text{if } p \leq 3, ~k = 1, \\
p \pmod{\mathfrak{p}^{2k+1}}, &\text{if } p > 3, \text{ or } p \leq 3, ~k \geq 2;
\end{cases}
\end{equation*}
and if $n > 1$, then
\begin{equation*}
\Phi_{p^k n}(\zeta) \equiv
\begin{cases}
1 \pmod{\mathfrak{p}^2}, &\text{if } p = 2, k =1 \\
1 \pmod{\mathfrak{p}^{2k+1}}, &\text{if } p = 2, ~k > 1, \\
1 \pmod{\mathfrak{p}^{2k-1}}, &\text{if } p > 2. 
\end{cases}
\end{equation*} 

If $p > 2$, we also have
\begin{equation*}
\beta^{\varphi(p^k n)} \equiv 1 \pmod{\mathfrak{p}^{2k-1}}
\end{equation*}
for all $k, n \geq 1$, by Fermat's little theorem for
\begin{equation*}
\mathcal{O}_\mathcal{K}/\mathfrak{p} \simeq \mathbb{Z}/p\mathbb{Z}.
\end{equation*}
When $n=1$, the congruence for $p=3$, $k=1$ cannot be improved upon in general; otherwise, if $p=3$, $k > 1$ or $p > 3$, $k \geq 1$, then it follows from the two congruences above that
\begin{equation*}
M^U_{p^k} \equiv p \pmod{\mathfrak{p}^{2k+1}},
\end{equation*}
so
\begin{equation*}
M^U_{p^k}-p \in \mathfrak{p}^{2k+1}~\cap~\mathbb{Z}=p^{k+1}\mathbb{Z}.
\end{equation*}
Similarly, if $n > 1$, then
\begin{equation*}
M^U_{p^k n}-1 \in \mathfrak{p}^{2k-1}~\cap~\mathbb{Z}=p^{k}\mathbb{Z}.
\end{equation*}

If $p=2$, $n = 1$ the congruences for the first term stabilize more slowly:
\begin{equation*}
\beta^{\varphi(2^k)} \equiv
\begin{cases}
1 \pmod{\mathfrak{p}}, &\text{if } k = 1, \\
1 \pmod{\mathfrak{p}^2}, &\text{if } k = 2, \\
1 \pmod{\mathfrak{p}^{2k-1}}, &\text{if } k \geq 3,
\end{cases}
\end{equation*}
leading to slightly different behavior in the congruences at small exponents, as in the theorem statement. On the other hand, if $n > 1$, then $(2,n)=1$ forces $n \geq 3$, so that
\begin{equation*}
\varphi(n) \equiv 0 \pmod{2},
\end{equation*}
from which it follows that
\begin{equation*}
\beta^{\varphi(n)} \equiv 1 \pmod{2\mathcal{O_K}},
\end{equation*}
and therefore
\begin{equation*}
\beta^{\varphi(2^k n)} \equiv 1 \pmod{\mathfrak{p}^{2k+1}}
\end{equation*}
for $k > 1$. The stronger congruences in this case then follow as in the previous argument.

\ref{mpk2} If $p \mid Q$, we necessarily have $\left(\frac{D}{p} \right)=1$, since $D=P^2-4Q$ and $p \nmid (P,Q)$ by hypothesis. Write $p\mathcal{O_K}=\mathfrak{p}\mathfrak{q}$. From $\alpha\beta = Q$, $\alpha+\beta=P$, we find that exactly one of $\alpha$, $\beta$ vanishes modulo $\mathfrak{p}$. Since
\begin{equation*}
\beta^{\varphi(n)}\Phi_n(\alpha/\beta) = \alpha^{\varphi(n)}\Phi_n(\beta/\alpha),
\end{equation*}
we assume without loss of generality that $\alpha \equiv 0 \pmod{\mathfrak{p}}$. Then
\begin{equation*}
    \beta^{\varphi(p^k)} \equiv 1 \pmod{\mathfrak{p}^k},~\zeta^{p^{k-1}} \equiv 0 \pmod{\mathfrak{p}^{k}},
\end{equation*}
from which it follows immediately that
\begin{equation*}
    M^U_{p^kn} \equiv 1 \pmod{\mathfrak{p}^k}.
\end{equation*}
Similarly,
\begin{equation*}
    M^U_{p^kn} \equiv 1 \pmod{\mathfrak{q}^k}.
\end{equation*}
By the Chinese Remainder Theorem,
\begin{equation*}
    M^U_{p^kn} \equiv 1 \pmod{p^k\mathcal{O_K}}.
\end{equation*}

If $p \nmid Q$, then it is clear that $z_u(p)=\text{ord}_{p\mathcal{O_K}}(\zeta)$ is the multiplicative order of $\zeta$ modulo $p\mathcal{O_K}$. Working with the Frobenius action
\begin{equation*}
    X \longmapsto X^p: \mathcal{O}_\mathcal{K}/p\mathcal{O}_\mathcal{K} \longrightarrow \mathcal{O}_\mathcal{K}/p\mathcal{O}_\mathcal{K}
\end{equation*}
on
\begin{equation*}
    \mathcal{O}_\mathcal{K}/p\mathcal{O}_\mathcal{K} \simeq
    \begin{cases}
        \mathbb{F}_p \times \mathbb{F}_p, &\text{if } \left(\frac{D}{p} \right) = 1, \\
        
        \mathbb{F}_{p^2}, &\text{if } \left(\frac{D}{p} \right) = -1,
    \end{cases}
\end{equation*}
we find that
\begin{equation*}
    \zeta^p \equiv \zeta^{\left(\frac{D}{p} \right)} \pmod{p\mathcal{O_K}}.
\end{equation*}
Incidentally, this proves the familiar fact that $z_U(p) \mid p - \left(\frac{D}{p} \right)$ for odd primes.

If $\left(\frac{D}{p} \right)=1$, note that the nontrivial $\mathbb{Q}$-automorphism of $\mathcal{K}$ acts as
\begin{equation*}
    \zeta + \mathfrak{p} \longmapsto\zeta^{-1}+ \mathfrak{q}: \mathcal{O_K}/\mathfrak{p} \longrightarrow \mathcal{O_K}/\mathfrak{q},
\end{equation*}
from which it follows immediately that
\begin{equation*}
    \text{ord}_{\mathfrak{p}}(\zeta)=\text{ord}_\mathfrak{q}(\zeta)=\text{ord}_{p \mathcal{O_K}}(\zeta) = z_U(p).
\end{equation*}
Since
\begin{equation*}
    \beta^{\varphi(p^kn)} \equiv 1 \pmod{p^k\mathcal{O_K}},
\end{equation*}
we can read off the congruences for $M^U_{p^kn}=\beta^{\varphi(p^kn)}\Phi_{p^kn}(\zeta)$ directly from Lemma \ref{ratcon-lem}, invoking the Chinese Remainder Theorem as necessary.

If $\left(\frac{D}{p} \right) = -1$, we adapt the arguments in Lemma \ref{ratcon-lem} as follows. Suppose first that $n \neq z_U(p)$, which implies that $\Phi_n(\zeta) \neq 0 \pmod{p\mathcal{O_K}}$. If $n = 1$, then from
\begin{equation*}
    \zeta \not\equiv 1 \pmod{p\mathcal{O_K}}, ~\zeta^{p^{k-1}} \equiv \zeta^{-p^k} \pmod{p^k\mathcal{O_K}},,
\end{equation*}
it follows easily that
\begin{equation*}
    \Phi_{p^k}(\zeta) \equiv -\zeta^{-p^{k-1}} \equiv -\zeta^{p^k} \pmod{p^k\mathcal{O_{K}}},
\end{equation*}
for all $k \geq 1$; therefore
\begin{align*}
    M^U_{p^k} &\equiv -\beta^{\varphi(p^k)}\zeta^{p^k} \\ &= -\beta^{-p^{k-1}}\alpha^{p^k} \\ &\equiv -\alpha^{-p^k}\alpha^{p^k} \\ &= -1 \pmod{p^k\mathcal{O_\mathcal{K}}}.
\end{align*}

If $n > 1$, we get
\begin{equation*}
    M^U_{p^kn} \equiv \beta^{\varphi(p^kn)} \Phi_n(\zeta^{p^k})\Phi_n(\zeta^{-p^k})^{-1} \pmod{p^k\mathcal{O_K}}.
\end{equation*}
Since
\begin{equation*}
    \Phi_n(X^{-1})=X^{-\varphi(n)}\Phi_n(X)
\end{equation*}
for $n > 1$, this reduces to
\begin{align*}
    M^U_{p^kn } &\equiv \beta^{\varphi(p^k n)}\zeta^{p^k \varphi(n)} \\ &= \beta^{-p^{k-1}\varphi(n)}\alpha^{p^k \varphi(n)} \\ &\equiv \beta^{-p^{k-1}\varphi(n)}\beta^{p^{k-1}\varphi(n)} \\ &\equiv 1 \pmod{p^k\mathcal{O_K}}.
\end{align*}

Suppose finally that $n = z_U(p)$; then necessarily $n > 1$. We have
\begin{equation*}
    n = \text{ord}_{p^{k+1}\mathcal{O_K}}\zeta^{p^k} = \text{ord}_{p^{k+1}\mathcal{O_K}}\zeta^{-p^k}
\end{equation*}
for all $k \geq 1$, from which we obtain polynomial congruences
\begin{equation*}
    \Phi_n(X) \equiv \prod_{(m,n)=1}(X-\zeta^{p^km}) \equiv \prod_{(m,n)=1}(X-\zeta^{-p^km}) \pmod{p^{k+1}\mathcal{O_K}[X]},
\end{equation*}
so
\begin{align*}
    \Phi_{p^kn}(X) &\equiv \prod_{(m,n)=1}\frac{X^{p^k}-\zeta^{-p^km}}{X^{p^k-1}-\zeta^{p^km}} \\
    &\equiv \prod_{(m,n)=1}\frac{X^{p^k}-\zeta^{p^{k+1}m}}{X^{p^k-1}-\zeta^{p^km}} \\
    &=\prod_{(m,n)=1}\sum_{j=0}^{p-1}\zeta^{p^kmj}X^{p^{k-1}(p-1-j)} \pmod{p^{k+1}\mathcal{O_K}[X]},
\end{align*}
the product running through any reduced residue system modulo $n$. Evaluating at $\zeta$ and using the facts that $-m$ runs through a reduced residue system modulo $n$ as $m$ does, and $p \equiv -1 \pmod{n}$, we find as in Lemma \ref{ratcon-lem} that
\begin{equation*}
    \Phi_{p^kn}(\zeta)\equiv\zeta^{\varphi(p^kn)}\prod_{(m,n)=1}\Phi_{p^k}(\zeta_m) \pmod{p^{k+1}\mathcal{O_K}},
\end{equation*}
where $\zeta_m=\zeta^{m-1}$. If $m \equiv 1 \pmod{n}$, we have
\begin{equation*}
    \Phi_{p^k}(\zeta_m) \equiv
    \begin{cases}
        0 \pmod{2\mathcal{O_K}}, &\text{if } p = 2, k = 1, \\
        p \pmod{p^{k+1}\mathcal{O_K}}, &\text{otherwise}.
    \end{cases}
\end{equation*}
This is already enough to complete the argument for $p = 2$, $k = 1$, so we assume from this point on that this is not the case.

If $n =2$, or in other words if $p \mid P$, then necessarily $p > 2$ (since $p \nmid D$ by hypothesis), and $\zeta \equiv -1 \pmod{p\mathcal{O_K}}$; then $m \equiv 1 \pmod{n}$ is the only contribution to the product, so that
\begin{equation*}
    M^U_{2p^k} \equiv \beta^{\varphi(p^k)}p \pmod{p^{k+1}\mathcal{O_K}}.
\end{equation*}
Since $\zeta \equiv -1 \pmod{p\mathcal{O_K}}$, therefore
\begin{equation*}
    -1 \equiv \zeta^{-p} = \beta^p \alpha^{-p} \equiv \beta^{p-1} \pmod{p\mathcal{O_K}},
\end{equation*}
so $\beta^{\varphi({p^k})} \equiv -1 \pmod{p^{k}\mathcal{O_K}}$, and therefore
\begin{equation*}
    M^U_{2p^k} \equiv -p \pmod{p^{k+1}}.
\end{equation*}

If $n > 2$, we need to work out the remaining terms in the product. For $m \not\equiv 1 \pmod{n}$, we have that $\text{ord}_{p\mathcal{O_K}}(\zeta_{m}) > 1$ divides $p +1$. Therefore
\begin{align*}
    \Phi_{p^k}(\zeta_m) &= \sum_{j=0}^{p-1}\zeta_m^{p^{k-1}j} \\
    &= \sum_{j=0}^{p+1}\zeta_m^{p^{k-1}j}-\zeta_m^{(p+1)p^{k-1}}-\zeta_m^{p^k} \\
    &\equiv -\zeta_m^{p^k} \pmod{p^k\mathcal{O_K}}.
\end{align*}
Obviously, $(-1)^{\varphi(n)-1}=-1$ for $n > 2$, so
\begin{align*}
    \Phi_{p^kn}(\zeta) &\equiv-\zeta^{\varphi(p^kn)}p\prod_{\substack{(m,n)=1 \\ m \not\equiv 1 \pmod{n}}}\zeta_m^{p^k} \\
    &=-\zeta^{\varphi(p^kn)+S}p \pmod{p^{k+1}\mathcal{O_K}},
\end{align*}
where
\begin{equation*}
    S = \sum_{\substack{(m,n)=1 \\m \not\equiv 1 \pmod{n}}}p^k(m-1) \equiv -p^k\varphi(n) \pmod{n};
\end{equation*}
we get
\begin{equation*}
    \Phi_{p^kn}(\zeta) \equiv -\zeta^{-p^{k-1}\varphi(n)}p \pmod{p^{k+1}\mathcal{O_K}},
\end{equation*}
and therefore
\begin{equation*}
    M^U_{p^kn} \equiv -\beta^{\varphi(p^kn)}\zeta^{p^{k}\varphi(n)}p \pmod{p^{k+1}\mathcal{O_K}}.
\end{equation*}
We have seen already that
\begin{equation*}
    \beta^{\varphi(p^kn)}\zeta^{p^k \varphi(n)} \equiv 1 \pmod{p^k\mathcal{O_K}},
\end{equation*}
so we are done.

\ref{mpk4} It is clear from the valuations in Proposition \ref{vpM-thm} for $p \mid (P,Q)$ that 
\begin{equation*}
v_p(M^U_{p^kn}) \geq p^{k-1}
\end{equation*}
for all $k \geq 1$ except possibly when $2v_p(P) > v_p(Q)$, $p=2$, $n = 1$, $k \geq 2$ in which case we have
\begin{equation*}
v_2(M^U_{2^k}) \geq 2^{k-2}+1.
\end{equation*}
\end{proof}

\begin{proof}[Proof of Theorem \ref{vmodpk-thm}]
The $p$-adic congruences for $M^V$ follow by a straightforward application of the doubling formula in Proposition \ref{dbl-thm} and the $p$-adic congruences for $M^U$ in the previous theorem.
\end{proof}

\begin{proof}[Proof of Corollary \ref{modn-thm}] 
\ref{modn1} It suffices to show that if $\partial_q(n) =z_u(q)$ for some prime $q \mid n$, then $q = p$, except in the listed cases. In fact, this is obvious: if $\partial_q(n) =z_u(q)$, then $\partial_{q}(n) \leq q+1$ (since $z_U(q) \mid q-\left(\frac{D}{q}\right)$), from which it follows immediately that either $q = p$, or else $q = 2$, $p = \partial_2(n) =3$.

\ref{modn2} Another straightforward application of the doubling formula.
\end{proof}

\begin{proof}[Proof of Corollary \ref{lift-thm}]
Applying the congruences in Theorems \ref{modpk-thm} and \ref{vmodpk-thm} to the identities
\begin{equation*}
    \frac{U_{p^k n}}{U_{p^{k-1}n}}= \prod_{d \mid n}M^U_{p^k d}, ~\frac{V_{p^k n}}{V_{p^{k-1}n}}= \prod_{d \mid n}M^V_{p^k d}
\end{equation*}
gives each congruence directly, with the exception of the last congruence
\begin{equation*}
    V_{p^kn} \equiv V_{p^{k-1}n} \pmod{p^k},
\end{equation*}
which is easily derived from the invariance of $\alpha+\beta$ under Frobenius actions in the Binet formula
\begin{equation*}
    V_n=\alpha^n+\beta^n.
\end{equation*}
\end{proof}

\begin{proof}[Proof of Corollary \ref{mult-cor}]
Suppose first that $\partial_p(n)=1$, that is, $n = p^k$ for some $k \geq 1$, and consider any prime $q \neq p$ dividing $M^U_{p^k}$. By Proposition \ref{vpM-thm}, it follows that $p^k = q^jz_U(q)$ for some $j \geq 0$; evidently, we must have $j = 0$, $z_U(q)=p^k$, so that also $v_q(M^U_{p^k})=v_q(U_{p^k})$. If $p \mid D$, then $v_p(M^U_{p^k})=1$ for $k \geq 2$, so that
\begin{equation*}
    \left|M^U_{p^k} \right| = p \times \prod_{z_U(q)=p^k}q^{v_q(U_{p^k})},
\end{equation*}
the product running over all characteristic factors of $U_{p^k}$. Moreover, $z_U(q)= p^k$ implies that $q \equiv \left(\frac{D}{q} \right) \pmod{p^k}$, so
\begin{equation*}
    \left|M^U_{p^k} \right| \equiv p \times \prod_{z_U(q)=p^k}\left(\frac{D}{q} \right)^{v_q(U_{p^k})} \pmod{p^k}.
\end{equation*}
Since
\begin{equation*}
    \prod_{z_U(q)=p^k}\left(\frac{D}{q} \right)^{v_q(U_{p^k})}= -1 \text{ or } 1,
\end{equation*}
and
\begin{equation*}
    M^U_{p^k} \equiv p \pmod{p^k}
\end{equation*}
by Theorem \ref{modpk-thm}, we find by considering solutions in $p$ prime, $k \geq 2$ to the congruence $p \equiv -p \pmod{p^k}$ that
\begin{equation*}
    \prod_{z_U(q)=p^k}\left(\frac{D}{q} \right)^{v_q(U_{p^k})}=\text{sign}(M^U_{p^k})
\end{equation*}
provided $k \geq 3$, if $p = 2$, or $k \geq 2$, if $p > 2$.

Similarly, if $p \nmid D$, we find that 
\begin{equation*}
    \left|M^U_{p^k} \right| \equiv  \prod_{z_U(q)=p^k}\left(\frac{D}{q} \right)^{v_q(U_{p^k})} \pmod{p^k}
\end{equation*}
and
\begin{equation*}
    M^U_{p^k}  \equiv \left(\frac{D}{p} \right) \pmod{p^k},
\end{equation*}
for all $k \geq 1$, from which we conclude by considering solutions in $p$ prime, $k \geq 1$ to the congruence $1 \equiv -1 \pmod{p^k}$ that
\begin{equation*}
    \prod_{z_U(q)=p^k}\left(\frac{D}{q} \right)^{v_q(U_{p^k})}=\left(\frac{D}{p}\right) \times \text{sign}(M^U_{p^k})
\end{equation*}
for all $k \geq 1$, except possibly if $p = 2$, $k = 1$.

If $\partial_p(n) > 1$, and $q \neq p$ divides $M^U_n$, then again $n = q^j z_U(q)$ for some $j \geq 0$; if $j > 0$, then we find, as in the proof of Corollary \ref{modn-thm}, that $p \leq z_U(q) \leq q+1$, in which case necessarily $q = 2$, $p = \partial_2(n) = 3$. In this case, we have, for $k \geq 3$, that
\begin{equation*}
    \left|M^U_{2^k \cdot 3} \right| \equiv 2 \times \prod_{z_U(q)=2^k \cdot 3}\left(\frac{D}{q} \right)^{v_q(U_{2^k \cdot 3})} \pmod{2^k \cdot 3}
\end{equation*}
and
\begin{equation*}
    M^U_{2^k \cdot 3} \equiv 2 \left(\frac{D}{2} \right) \pmod{2^k \cdot 3};
\end{equation*}
otherwise, if $p > 3$ or $9 \mid n$, then
\begin{equation*}
    \left|M^U_n \right| \equiv
    \begin{cases}
        p \times \prod\limits_{z_U(q)=n}\left(\frac{D}{q} \right)^{v_q(U_n)} &\pmod{n}, \text{ if } \partial_p(n)=z_U(p), \\
        \prod\limits_{z_U(q)=n}\left(\frac{D}{q} \right)^{v_q(U_n)} &\pmod{n}, \text{ otherwise},
    \end{cases}
\end{equation*}
and
\begin{equation*}
    M^U_n \equiv
    \begin{cases}
        \left(\frac{D}{p} \right)p &\pmod{n},  \text{ if } \partial_{p}(n) = z_U(p), \\
        1 &\pmod{n}, \text{ otherwise},
    \end{cases}
\end{equation*}
by Theorem \ref{modn-thm}. In either case, the result follows by the same analysis as in the previous paragraph.
\end{proof}

%--------------------5.Bibliography---------------------------------------

\end{document}